\documentclass[11pt, a4paper]{amsart}
\usepackage{amsfonts,amscd,latexsym,amsmath,amssymb,cancel,enumerate,mathrsfs,hyperref,color}
\theoremstyle{plain}
\newtheorem{master}{Master}[section]
\newtheorem{prop}[master]{Proposition}
\newtheorem{thm}[master]{Theorem}
\newtheorem{fact}[master]{Fact}
\newtheorem{lem}[master]{Lemma}
\newtheorem{cor}[master]{Corollary}
\newtheorem{question}[master]{Question}
\newtheorem{claim}[master]{Claim}
\theoremstyle{definition}
\newtheorem{defin}[master]{Definition}

\theoremstyle{remark}
\newtheorem{remark}[master]{Remark}
\numberwithin{equation}{section}

\newcommand{\Rea}{\mathbb{R}}
\newcommand{\Nat}{\mathbb{N}}
\newcommand{\Rat}{\mathbb{Q}}
\newcommand{\Int}{\mathbb{Z}}

\newcommand{\dist}{\mathrm{dist}}
\newcommand{\Gr}{\mathbb{G}}

\begin{document}
\title{Generic norms and metrics on countable abelian groups}
\author[M. Doucha]{Michal Doucha}
\address{Laboratoire de Math\' ematiques de Besan\c con\\Universit\' e de Franche-Comt\' e\\France} 
\address{Institute of Mathematics CAS, \v Zitn\' a 25, 115 67 Prague, Czech Republic}
\email{doucha@math.cas.cz}

\begin{abstract}
For a countable abelian group $G$ we investigate generic properties of the space of all invariant metrics on $G$. We prove that for every such an unbounded group $G$, i.e. group which has elements of arbitrarily high order, there is a dense set of invariant metrics on $G$ which make $G$ isometric to the rational Urysohn space, and a comeager set of invariant metrics such that the completion is isometric to the Urysohn space. This generalizes results of Cameron and Vershik, Niemiec, and the author.

Then we prove that for every countable abelian $G$ such that $G\cong \bigoplus_\Nat G$ there is a comeager set of invariant metrics on $G$ such that all of them give rise to the same metric group after completion. If moreover $G$ is unbounded, then using a result of Melleray and Tsankov we get that the completion is extremely amenable.
\end{abstract}
\maketitle
\section*{Introduction}
For an unbounded countable abelian group $G$, J. Melleray and T. Tsankov prove in \cite{MeTs} that the set of all invariant metrics on $G$ which make $G$ extremely amenable is comeager in the Polish space of all invariant metrics on $G$. That result motivated the work presented in this paper. We focus on two main themes:
\subsection*{Groups isometric to the Urysohn space} In \cite{CaVe}, Cameron and Vershik prove that there is an invariant metric on $\Int$ which makes it isometric to the rational Urysohn space. In particular, the completion is isometric to the Urysohn space and thus the Urysohn space has a structure of a monothetic abelian group. Niemiec in \cite{Nie1} proves that the Shkarin's universal metric abelian Polish group is isometric to the Urysohn space and its canonical countable dense subgroup (which is $\bigoplus_\Nat \Rat/\Int$) is isometric to the rational Urysohn space. Here we generalize these results by proving that actually every unbounded countable abelian group, i.e. group having elements of arbitrarily high order, admits an invariant metric which makes it isometric to the rational Urysohn space. In particular, this covers the cases of $\Int$ and $\bigoplus_\Nat \Rat/\Int$. We prove something more general stated in the theorem below.
\begin{thm}\label{intro_thm1}
For any unbounded countable abelian group $G$, the set of all invariant metrics which make $G$ isometric to the rational Urysohn space is dense in the space of all invariant metrics on $G$. Moreover, the set of all invariant metrics on $G$ with which the completion is isometric to the Urysohn space is dense $G_\delta$.
\end{thm} 
\subsection*{Generic metrics} We motivate the next topic by the following facts. We recall that in the Polish space of all countable graphs, those that are isomorphic to the random graph form a comeager subset. One cannot expect literally the same property for a metric structure, however it is true that in the Polish space of all metrics on a countable set, those whose completion is isometric to the Urysohn universal space form a comeager set. Analogously, one can show that in the Polish space of all norms on the countable infinite-dimensional vector space over $\Rat$, those whose completion is isometric to the Gurarij space (\cite{Gu}) form a comeager set.

Our aim is to find generic metrics on countable abelian groups. There are of course many countable abelian groups to consider. However, we can prove the following result.
\begin{thm}\label{intro_thm2}
Let $G$ be a countable abelian group such that $G\cong \bigoplus_\Nat G$. Then there is a comeager set of invariant metrics on $G$ such that all of them give rise to the same metric group after the completion.
\end{thm}

When the group $G$ is moreover unbounded we get the following.
\begin{cor}
Let $G$ be an unbounded countable abelian group such that $G\cong \bigoplus_\Nat G$. Then there is a comeager set of invariant metrics on $G$ such that all of them give rise to the same metric group after the completion which is isometric to the Urysohn space and extremely amenable.
\end{cor}
We also derive the extreme amenability of a universal abelian Polish group.
\begin{cor}
The universal abelian Polish group of Shkarin \cite{Sk} and Niemiec \cite{Nie1} is extremely amenable.
\end{cor}

\section{Definitions and preliminaries}
Throughout this paper, since we shall work only with abelian groups, we adopt the additive notation.

A metric $d$ on an abelian group $G$ is invariant if for every $a,b,c\in G$ we have $d(a,b)=d(a+c,b+c)$. A norm on $G$ is a function $\lambda:G\rightarrow \Rea_0^+$ from $G$ to non-negative reals, denoted by $\Rea_0^+$, which attains zero only at $0_G\in G$ and satisfies for every $a,b\in G$, $\lambda(a)=\lambda(-a)$ and $\lambda(a+b)\leq \lambda(a)+\lambda(b)$. There is a one-to-one correspondence between invariant metrics and norms on abelian groups: for a norm $\lambda$, $d_\lambda(a,b):=\lambda(a-b)$ defines an invariant metric, and for an invariant metric $d$, $\lambda(a):=d(a,0)$ defines a norm.

Consider the set of all norms, respectively invariant metrics, on some countable abelian group $G$. It can be viewed as a subset of $\Rea^G$, resp. $\Rea^{G\times G}$. In both cases, one can easily check it is a closed set, thus a Polish space (we refer the reader to \cite{Ke} for facts needed about Polish spaces). It turns out it is more convenient for us to work with norms, rather than invariant metrics, so we shall do so mostly in the sequel.

Let us denote the Polish space of norms on $G$ by $\mathcal{N}_G$. Later, when the group is known from the context, or it is fixed, we shall just write $\mathcal{N}$. Clearly, the space $\mathcal{N}_G$ is homeomorphic with the Polish space of all invariant metrics on $G$ via the formula above.

The content of this paper is to investigate generic properties of the space of norms on countable abelian groups. We recall one important result of Melleray and Tsankov in that direction that we shall apply in our paper. Let us first state the following simple, but relevant for us, definition.

\begin{defin}
We shall call an abelian group $G$ \emph{unbounded} if it either contains an element of infinite order or it contains elements of arbitrarily high finite orders.
\end{defin}

\begin{thm}[Melleray, Tsankov, Theorem 6.4 in \cite{MeTs}]\label{thmMeTs}
Let $G$ be a countable unbounded abelian group. Then the set $\mathbb{E}=\{\lambda\in\mathcal{N}_G:(G,\lambda)\text{ is extremely amenable}\}$ is dense $G_\delta$ in $\mathcal{N}_G$.
\end{thm}
We note that Melleray and Tsankov formulated the theorem with invariant metrics rather than norms which is, as noted above, however equivalent.\\
\begin{defin}
Let $G$ be an abelian group and $A\subseteq G$ a symmetric subset, i.e. $A=-A$, containing zero. A \emph{partial norm} on $A$ is a function $\lambda:A\rightarrow \Rea^+_0$ satisfying the following requirements:
\begin{itemize}
\item $\lambda(x)=0$ iff $x=0$, for $x\in A$,
\item $\lambda(x)=\lambda(-x)$ for $x\in A$,
\item $\lambda(x)\leq \sum_{i=1}^n \lambda(x_i)$, where $n$ is arbitrary, $x,x_1,\ldots,x_n\in A$ and $x=\sum_{i=1}^n x_i$.

\end{itemize}
If $\lambda$ satisfies all the conditions except the first one, then it is called a \emph{partial seminorm}. If it satisfies the first and the second condition we call it a \emph{partial pre-norm}.
\end{defin}
\begin{lem}\label{extnorm_lem}
Let $G$ be an abelian group, $A\subseteq G$ some symmetric subset containing zero and $\lambda_A:A\rightarrow\Rea_0^+$ a partial norm on $A$. Then for any subset $B$ with $A\subseteq B\subseteq \langle A\rangle\leq G$, where $\langle A\rangle$ is the subgroup of $G$ generated by $A$, there exists a partial seminorm $\lambda_B:B\rightarrow \Rea^+_0$ which extends $\lambda_A$.

Moreover, if $A$ is finite and $B=\langle A\rangle=G$, then $\lambda_B$ is a norm on $G$ extending $\lambda_A$, and if $\lambda_A$ is additionally rational-valued, then so is $\lambda_B$.
\end{lem}
\begin{proof}
Take $b\in B$ and set $$\lambda_B(b)=\inf \{\sum_{i=1}^n \lambda_A(a_i):(a_i)_{i=1}^n\subseteq A,b=\sum_{i=1}^n a_i\}.$$
Since $B\subseteq \langle A\rangle$, there exist $(a_i)_{i=1}^n\subseteq A$ such that $b=\sum_{i=1}^n a_i$. It directly follows from the definition that $\lambda_B$ satisfies all the conditions of a partial seminorm and extends $\lambda_A$. Moreover, if $A$ is finite, then the infimum from the definition of $\lambda_B$ might be replaced by the minimum. It follows that $\lambda_B$ in that case is a partial norm and if $\lambda_A$ was rational-valued, so is now $\lambda_B$.
\end{proof}
\begin{remark}
Notice that $\lambda_B$ is actually the greatest extension of $\lambda_A$ to a partial seminorm on $B$; i.e. if $\lambda:B\rightarrow \Rea^+_0$ is any partial seminorm on $B$ that extends $\lambda_A$, then $\lambda\leq \lambda_B$; and if $A$ is finite, then $\lambda_B$ is the greatest partial norm extending $\lambda_A$.
\end{remark}
When we are given a symmetric subset $A\subseteq G$ containing zero of some abelian group and also some partial pre-norm $\rho':A\rightarrow \Rea_0^+$ then we can get, using the same formula as in the proof above, a greatest partial seminorm determined by $\rho'$; or, in the case $A$ is finite, a greatest partial norm determined by $\rho'$. We state that explicitly in the next fact and omit its easy proof.
\begin{fact}\label{getnormfact}
Let $A\subseteq G$ be a symmetric subset containing zero of some abelian group. Let $\rho':A\rightarrow \Rea_0^+$ be a pre-norm. Then the formula, applied for every $x\in A$, $$\rho(x)=\inf\{\sum_{i=1}^n \rho'(x_i):x=\sum_{i=1}^n x_i, (x_i)_{i-1}^n\subseteq A\}$$ gives a partial seminorm on $A$, resp. partial norm on $A$ if $A$ is finite.
\end{fact}

In the text below we shall be interested in the possibilities how to extend partial norms. We recall the notion of a Kat\v etov function on a metric space which corresponds, in the terminology of continuous model theory, to the quantifier-free type over a metric space. Let $X$ be a metric space and $f:X\rightarrow \Rea^+_0$ a function. It is called \emph{Kat\v etov} if it satisfies for every $x,y\in X$ $$|f(x)-f(y)|\leq d_X(x,y)\leq f(x)+f(y).$$ A Kat\v etov function $f$ without zeros can be then viewed as a prescription of distances of a new point to the points of $X$ in the sense that we can define a one-point extension $X\cup\{x_f\}$ and define the distance of $x_f$ to a point $y\in X$ as $f(y)$.\\

We start with some algebraic definition of elements in abelian groups that have the potential to realize Kat\v etov functions. Fix $G$ an abelian group and $A\subseteq G$ a subset. We recall that \emph{an oriented Cayley graph} $C_G^A$ of $G$ with respect to the set $A$ is an oriented graph such that
\begin{itemize}
\item the set of vertices is $G$;
\item the set of oriented edges is the set $\{(g,g+a):g\in G,a\in A\}$.
\end{itemize}
\emph{An oriented path} from $g\in G$ to $h\in G$ is a sequence $g=x_0,\ldots,x_n=h$ from $G$ such that for every $0\leq i<n$, $(x_i,x_{i+1})$ is an oriented edge. In other words, there are elements $a_1,\ldots,a_n\in A$ such that $x_i=x_0+\sum_{j=1}^i a_j$, for $1\leq j\leq n$. The length of this path is $n$.
\begin{defin}
Let now $G$ be again an abelian group and $A\subseteq G$ a symmetric subset (containing zero). Let $g\in G\setminus A$. We define the \emph{distance} of $g$ from $A$ in $G$, denoted by $\dist_G(g,A)$, as the length of the shortest oriented path between $g$ and $A$ in $C_G^{A\cup\{g\}}$.
\end{defin}
Now let $G$ be an abelian group, $A\subseteq G$ a finite symmetric subset containing zero and $\lambda_A$ a partial norm on $A$. Set $\overline{A}=\{a-b:a,b\in A\}=A-A=A+A$. Note that $A\subseteq \overline{A}$ and that $\overline{A}$ is again a finite symmetric subset containing zero. Let $\bar{\lambda}_A$ be the greatest extension of $\lambda_A$ into a partial norm on $\overline{A}$ guaranteed by Lemma \ref{extnorm_lem}, and note that it induces a metric $d_A$ on $A$ defined as $d_A(a,b)=\bar{\lambda}_A(a-b)$ for $a,b\in A$. Let $f:A\rightarrow \Rea^+$ be a Kat\v etov function with respect to the metric $d_A$. We are now interested whether it is possible to find an element $g\in G\setminus A$ and a partial norm $\lambda$ on $\overline{A\cup \{g\}}$ which extends $\bar{\lambda}_A$ and such that for every $a\in A$ we have $f(a)=\lambda(g-a)$.

Set $m=\min\{\min \bar{\lambda}_A(\overline{A}\setminus\{0\}),\min f(A)\}$ and\\ $M=\max\{\max \bar{\lambda}_A(\overline{A}),\max f(A)\}$. Then we have the following proposition.
\begin{prop}\label{Katetovext_prop}
Under the setting above, suppose that there exists $g\in G\setminus \overline{A}$ such that $\dist(g,\overline{A})>2\frac{M}{m}$. Then there exists a partial norm $\lambda$ on $\{g-a,a-g:a\in A\}\cup \overline{A}$ which extends $\bar{\lambda}_A$ and such that for every $a\in A$ we have $f(a)=\lambda(g-a)$.
\end{prop}
\begin{proof}
Set $B=\overline{A}\cup\{g-a,a-g:a\in A\}$. Note that $B$ is also a symmetric set containing zero. Define $\lambda$ on $B$ as follows: for any $b\in B$ set $$\lambda(b)=\begin{cases} \bar{\lambda}_A(b) & b\in \overline{A},\\
f(x) & \text{if }b=g-x,\text{ or }b=x-g.\\
\end{cases}$$

Notice that there is no collision in the definition above since the sets $\overline{A}$ and $\{g-a,a-g:a\in A\}$ are disjoint. Indeed, if for some $a\in A$ and $\bar{a}\in\overline{A}$ we have $g-a=\bar{a}$, then there is an oriented edge between $g$ and $\bar{a}$ in $C_G^{\overline{A}\cup\{g\}}$, and that is a contradiction with $\dist(g,\overline{A})>2\frac{M}{m}\geq 2$.

Now it suffices to check that $\lambda$ is a partial norm on $B$. The first two conditions of the definition of a partial norm are easily checked. We claim that also the last condition is satisfied. Suppose otherwise. Then there are $b\in B$ and $(b_i)_{i=1}^n\subseteq B\setminus\{0\}$ such that $b=\sum_{i=1}^n b_i$ and $\lambda(b)>\sum_{i=1}^n \lambda(b_i)$. We have that $n<\frac{M}{m}$ because $M\geq \lambda(b)>n\cdot m$.

Moreover, without loss of generality we may also suppose that for no $i\neq j\leq n$ we have $b_i=g-a_i$ and $b_j=a_j-g$, for some $a_i,a_j\in A$. Indeed, suppose otherwise. Then since $g-a_i+a_j-g=a_j-a_i\in \overline{A}$ and the function $f$ is Kat\v etov we have that $\lambda(g-a_i)+\lambda(a_j-g)=f(a_i)+f(a_j)\geq d_A(a_j,a_i)=\bar{\lambda}_A(a_j-a_i)$. Thus we can replace the pair $b_i,b_j$ by $a_j-a_i$.\\

\noindent {\bf Case 1:} $b\in \overline{A}$. We claim that there must be some $i\leq n$ such that $b_i$ is equal to $g-a$ or $a-g$ for some $a\in A$. Indeed, otherwise we get into a contradiction since $\bar{\lambda}_A$ is a partial norm. By the argument in the paragraph above there is a single sign $\varepsilon\in\{1,-1\}$ such that for every $b_i\notin \overline{A}$ we have $b_i=\varepsilon\cdot g-\varepsilon\cdot a_i$ for some $a_i\in A$. Suppose that $\varepsilon=1$, the other case is analogous. It follows that there is some $0<k\leq n$ such that $b=k\cdot g+\sum_{i=1}^n c_i$, where $(c_i)_{i=1}^n\subseteq \overline{A}$. Thus there is an oriented path from $g$ to $b$ in $C_G^{\overline{A}\cup\{g\}}$ of length $k-1+n< 2n$ which contradicts that $\dist (g,\overline{A})>2\frac{M}{m}>2n$. Indeed, set $e_0=0$ and $e_i=g$, for $1\leq i\leq k-1$, and $e_j=c_{j-k+1}$, for $k\leq j\leq k-1+n$. Then $(x_i=g+\sum_{j=0}^i e_j)_{i=0}^{k-1+n}$ is the desired path.\\

\noindent {\bf Case 2:} $b=g-a$, or $b=a-g$, for some $a\in A$. Let us say $b=g-a$, the other case is analogous. Suppose at first that for all $i\leq n$ we have $b_i\in \overline{A}$. Since $g-\sum_{i=1}^n b_i=a$, if we set $x_0=g$ and $x_j=g-\sum_{i=1}^j b_i$, for $1\leq j\leq n$, we get that $x_0,\ldots,x_n$ is an oriented path from $g$ to $a$ in $C_G^{\overline{A}\cup\{g\}}$ which is again a contradiction with $\dist(g,\overline{A})>2\frac{M}{m}$.

Thus we suppose that for some $i\leq n$ we have $b_i\notin \overline{A}$. As in Case 1 we may suppose that there is a single sign $\varepsilon\in\{1,-1\}$ such that for every $b_i\notin \overline{A}$ we have $b_i=\varepsilon\cdot g-\varepsilon\cdot a_i$ for some $a_i\in A$.

Suppose at first that $\varepsilon=-1$. Then there is some $1<k\leq n+1$ such that $k\cdot g=\sum_{i=1}^n c_i$, where $(c_i)_{i=1}^n\subseteq \overline{A}$. Thus we get an oriented path of length $k-1+n-1< 2n$ from $g$ to $c_n$ in $C_G^{\overline{A}\cup\{g\}}$ which again contradicts that $\dist (g,\overline{A})>2\frac{M}{m}$.

Now suppose that $\varepsilon=1$. Suppose without loss of generality that $b_1=g-a_1$, for some $a_1\in A$. Then, since $n>1$, $\sum_{i=2}^n b_i=a_1-a\in\overline{A}$. If there are indices $2\leq i\leq n$ such that $b_i=g-a_i$, for some $a_i\in A$, then there is some $0<k\leq n-1$ so that $a_1-a=k\cdot g+\sum_{i=1}^{n-1} c_i$, where $(c_i)_{i=1}^{n-1}\subseteq \overline{A}$. As above, this again leads to a contradiction with $\dist(g,\overline{A})> 2\frac{M}{m}$.

Thus suppose that for every $2\leq i\leq n$ we have $b_i\in\overline{A}$. Then we claim that we may suppose that $n=2$. Indeed, we have that $\sum_{i=2}^n b_i=a_1-a\in\overline{A}$ and thus $\bar{\lambda}_A(a_1-a)\leq \sum_{i=2}^n \bar{\lambda}_A(b_i)$. So we are left with the case that $$\lambda(g-a)=f(a)>\lambda(g-a_1)+\lambda(a_1-a)=f(a_1)+\bar{\lambda}_A(a_1-a)=$$ $$f(a_1)+d_A(a_1,a),$$ which is again a contradiction with the fact that $f$ is Kat\v etov.
\end{proof}

\begin{lem}\label{unbddgrp_lem}
Let $G$ be an unbounded abelian group, $A\subseteq G$ a finite symmetric subset containing zero and $R>0$ a real number. Then there exists $g\in G$ such that $\dist(g,A)>R$.
\end{lem}
\begin{proof}
Set $B=\{n_1\cdot a_1+\ldots+n_i\cdot a_i:a_1,\ldots,a_i\in A,n_1,\ldots,n_i\geq 0,n_1+\ldots+n_i\leq R\}$. Note that $B$ is again finite symmetric and containing zero. It suffices to show that there is $g\in G$ such that $n\cdot g\notin B$, for every $0< n\leq R$. Indeed, suppose we have found such $g\in G$, yet still $\dist(g,A)\leq R$. Then there exists a sequence $(c_i)_{i=1}^n\subseteq A\cup\{g\}$, where $n\leq R$, such that $g+\sum_{i=1}^n c_i\in A$. Without loss of generality, we may suppose that there is some $1\leq j\leq n+1$ such that $c_i=g$ if and only if $i<j$. Then $b=\sum_{i=j}^n c_i\in B$ and $j\cdot g=-b$, a contradiction.

If there is $b\in B$ such that $b$ has infinite order, then it clearly suffices to take $N\cdot b$, for $N$ sufficiently large, as $g$. So suppose that every $b\in B$ has finite order. Let $N$ be the maximum of orders of elements from $B$. Suppose there is no such $g\in G$, thus for every $g\in G$ there are $n\leq R$ and $b\in B$ such that $n\cdot g=b$. However, then the order of every $g\in G$ is bounded by $R\cdot N$. That is a contradiction with unboundedness of $G$.
\end{proof}
\section{Groups isometric to the rational Urysohn space}
We recall here that the Urysohn universal metric space $\mathbb{U}$ is the unique Polish metric space (i.e. complete and separable) containing isometrically every separable metric space (containment of every finite metric space is enough) and satisfying the property that a partial isometry between two finite subsets extends to an autoisometry of the whole space. It was constructed by Urysohn in \cite{Ur}. We refer to Chapter 5 in \cite{Pe} for more information about this space.

The Urysohn space has one distinguished countable dense set which is called \emph{the rational Urysohn space} and denoted by $\Rat\mathbb{U}$. It is the unique countable metric space with rational distances that contains isometrically every finite rational metric space and again has the property that any partial isometry between two finite subsets extends to an autoisometry of the whole space. The following well-known fact gives another characterization of $\Rat\mathbb{U}$ which we will use.
\begin{fact}\label{isUryfact}
Let $X$ be a countable metric space with rational distances. Then $X$ is isometric to $\Rat\mathbb{U}$ iff for every finite subset $F\subseteq X$ and every rational Kat\v etov function $f:F\rightarrow \Rat$ there exists $x\in X$ realizing $f$, i.e. ${\forall y\in F (d(x,y)=f(y))}$.
\end{fact}
\begin{lem}\label{ratpartnorm}
Let $G$ be an abelian group, $F\subseteq G$ a finite symmetric subset containing zero and let $\rho$ be a partial norm on $F$. Then for any $\varepsilon>0$ there exists a rational partial norm $\rho_R$ on $F$ such that  for any $f\in F$ we have $|\rho(f)-\rho_R(f)|<\varepsilon$.
\end{lem}
\begin{proof}
Enumerate $F$ as $(f_i)_{i=1}^n$ such that for $i<j\leq n$ we have $\rho(f_i)\geq \rho(f_j)$, and $f_n=0$. Let $\delta=\min\{\varepsilon,\min\{|\rho(f)-\rho(g)|:f,g\in F,\rho(f)\neq\rho(g)\}\}$. Moreover, choose an (not necessarily strictly) increasing sequence of positive real numbers $(r_i)^{n-1}_{i=1}$ such that
\begin{itemize}
\item for any $i,j<n$, $\rho(f_i)=\rho(f_j)$ if and only if $r_i=r_j$;
\item for any $i<n$ we have $r_i<\delta$ and $\rho(f_i)+r_i\in\Rat$.
\end{itemize}
Now for $i<n$ we define $\rho_R(f_i)=\rho(f_i)+r_i$, and $\rho_R(f_n)=0$. We claim $\rho_R$ is as desired. Clearly, it is rational and for any $i\leq n$ we have $\varepsilon>r_i=\rho_R(f_i)-\rho(f_i)\geq 0$. If we check that $\rho_R$ is a partial norm then we will be done.

First, for any $i\leq n$ we have $\rho_R(f_i)=0$ if and only if $f_i=f_n=0$, and also for any $i<j<n$, if $f_i=-f_j$, then $\rho_R(f_i)=\rho_R(f_j)$. So it remains to check that for any $i_1,\ldots,i_k,i< n$ such that $f_i=f_{i_1}+\ldots+f_{i_k}$ we have $\rho_R(f_i)\leq \rho_R(f_{i_1})+\ldots+\rho_R(f_{i_k})$. Notice that we have $\rho_R(f_1)\geq \ldots\geq\rho_R(f_n)=0$ since for $j<l\leq n$, $\rho(f_j)=\rho(f_l)$ if and only if $\rho_R(f_j)=\rho_R(f_l)$, and otherwise we have $\rho(f_j)>\rho(f_l)$, so $\rho_R(f_l)=\rho(f_l)+r_l\leq \rho(f_j)<\rho_R(f_j)$. So if for some $j\leq k$ we have that $i_j\leq i$, then $\rho_R(f_i)\leq \rho_R(f_{i_j})$ and we are done. Otherwise, for all $j\leq k$ we have $i<i_j$. Then since $\rho_R(f_i)-\rho(f_i)=r_i\leq r_{i_1}=\rho_R(f_{i_1})-\rho(f_{i_1})$, so $$\rho_R(f_i)=\rho(f_i)+r_i\leq r_{i_1}+\sum_{j=1}^k \rho(f_{i_j})\leq \sum_{j=1}^k \rho_R(f_{i_j}).$$
\end{proof}
\begin{thm}\label{Urysohnnorms}
Let $G$ be a countable unbounded abelian group. Then there exists a norm $\lambda$ on $G$ such that $(G,\lambda)$ is isometric to the rational Urysohn space.

Moreover, the set of all norms on $G$ which make $G$ isometric to the rational Urysohn space is dense.
\end{thm}
\begin{proof}
The proof uses Lemmas \ref{unbddgrp_lem} and \ref{extnorm_lem}, and Proposition \ref{Katetovext_prop} and follows the standard construction of the rational Urysohn space. Enumerate $G$ as $\{g_n:n\in\Nat\}$ and let $\{(A_i,f_i):i\in\Nat\}$ be an enumeration with infinite repetition of all pairs $(A,f)$, where $A$ is a finite rational metric space and $f:A\rightarrow \Rat$ a rational Kat\v etov function over $A$.\\

By induction, we shall produce finite symmetric sets $F_i$, $i\in\Nat_0$, containing zero, with partial rational norm $\lambda_i$ on $F_i$ such that for every $n$, $F_n\subseteq F_{n+1}$ and $\lambda_n\subseteq \lambda_{n+1}$,  $G=\bigcup_n F_n$, and such that $G$ with the metric induced by the norm $\lambda=\bigcup_n \lambda_n$ is isometric to the rational Urysohn space.\\

Set $F_0=\{0\}$ and let $\lambda_0$ be the trivial norm on $F_0$. Now suppose that for some even $n$, $F_n$ and $\lambda_n$ on $F_n$ have been defined. We define $F_{n+1}$ and $\lambda_{n+1}$. Take the element $g=g_{n/2+1}$. If $g\in F_n$ then we do nothing, i.e. set $F_{n+1}=F_n$ and $\lambda_{n+1}=\lambda_n$. So suppose that $g\notin F_n$. We set $F_{n+1}=F_n\cup\{g,-g\}$. We need to extend $\lambda_n$. We distinguish cases:
\begin{itemize}
\item $g\in \langle F_n\rangle$: then we use Lemma \ref{extnorm_lem} to extend the partial norm $\lambda_n$ on $F_n$ to a partial norm $\lambda_{n+1}$ on $F_{n+1}$,
\item $g\notin \langle F_n\rangle$: then we can set 
\begin{enumerate}
\item $\lambda_{n+1}(a)=\lambda_n(a)$ if $a\in F_n$,
\item let $m$ be the minimal positive integer such that for some $f\in F_n\setminus\{0\}$, $m\cdot g=f$; if no such $m$ exists (then in particular, $g$ has an infinite order), or $f=0$, then we set $\lambda_{n+1}(g)=\lambda_{n+1}(-g)=1$; otherwise, we set $\lambda_{n+1}(g)=\lambda_{n+1}(-g)=\lambda_n(f)/m$.
\end{enumerate}
In any case, it is easy to check that $\lambda_{n+1}$ is a partial norm on $F_{n+1}$.
\end{itemize}

Now suppose that for some odd $n$, $F_n$ and $\lambda_n$ on $F_n$ have been defined. Set $G_n=\overline{F}_n=\{a-b:a,b\in F_n\}$ and extend $\lambda_n$ to $\rho_n$ on $G_n$ by Lemma \ref{extnorm_lem}. Then $\rho_n$ induces a metric $d_n$ on $F_n$, as usual, by $d_n(a,b)=\rho_n(a-b)$, for $a,b\in F_n$.

Set $(A,f)=(A_{\frac{n+1}{2}},f_{\frac{n+1}{2}})$. If there is no subset of $(F_n,d_n)$ isometric to $A$ then do nothing and set $F_{n+1}=F_n$ and $\lambda_{n+1}=\lambda_n$. Otherwise, take some subset $B\subseteq F_n$ isometric to $A$ and consider $f$ to be defined on $B\cong_{\mathrm{iso}} A$. We can clearly extend $f$ to the whole $F_n$, still denoted by $f$, so that it is still Kat\v etov and rational-valued. Just set for instance $f(x)=\min\{f(a)+d_n(a,x):a\in B\}$, for $x\in F_n$.

Set $m=\min\{\min \rho_n(G_n\setminus\{0\}), \min f(F_n)\}$ and\\ $M=\max\{\max \rho_n(G_n), \max f(F_n)\}$. By Lemma \ref{unbddgrp_lem} we can find some element $g\in G$ such that $\dist(g,G_n)>2\frac{M}{m}$. Then by Proposition \ref{Katetovext_prop} we can extend $G_n$ to $F'_{n+1}=G_n\cup\{g-a,a-g:a\in F_n\}$ and $\rho_n$ to a partial norm $\lambda'_{n+1}$ on $F'_{n+1}$ such that $f(x)=\lambda'_{n+1}(g-x)$ for every $x\in F_n$. If there was a unique isometric embedding of $A$ into $(F_n,d_n)$ then we are done. Otherwise, we consecutively repeat the above procedure for all other isometric embeddings of $A$ into $F_n$ enlarging $F'_{n+1}$ to $F''_{n+1}$, then to $F'''_{n+1}$, etc. The last obtained set is $F_{n+1}$.\\

When the induction is finished we have that $G=\bigcup_n F_n$ since at the $n$-th step, for even $n$, we have guaranteed that $g_{n/2+1}$ is contained in $F_{n+1}$. Moreover, $G$ with the metric induced by the norm $\lambda=\bigcup_n \lambda_n$ is isometric to the rational Urysohn space. By Fact \ref{isUryfact} it suffices to check it satisfies the rational one-point extension property. However if we take some finite $B\subseteq G$ and one-point extension determined by a rational Kat\v etov function $f$ on $B$, then we can find $n$ such that $B\subseteq F_{2n-1}$, $B$ is isometric to $A_n$ and $f$ on $B$ corresponds to $f_n$ on $A_n$. Then we have guaranteed that the Kat\v etov function is realized in $F_{2n}$.\\

Finally, we show how to get the ``moreover part" from the statement of the theorem, i.e. that the set of norms with which $G$ is isometric to the rational Urysohn space is actually dense.

Take any finite symmetric subset $F\subseteq G$  containing zero, let $\rho$ be an arbitrary partial norm on $F$ and $\varepsilon>0$ arbitrary. Then using Lemma \ref{ratpartnorm} we get a partial rational norm $\rho_R$ on $F$ such that for every $f\in F$ we have $|\rho(f)-\rho_R(f)|<\varepsilon$. We just set $F_0=F$ and $\lambda_0=\rho_R$. Then we continue the induction as above and obtain at the end a norm $\lambda$ with which $G$ is isometric to the rational Urysohn space and is $\varepsilon$-close on a finite subset $F$ to the partial norm $\rho$.
\end{proof}
\begin{cor}\label{corollaryUry}
If $G$ is unbounded, then the set of norms\\ $\{\lambda: \overline{(G,\lambda)}\text{ is isometric to }\mathbb{U}\}$ is comeager.
\end{cor}
\begin{proof}
Consider the set of norms $\lambda$ on $G$ satisfying the following condition:
\begin{equation}\label{Uryaftercompl}
\begin{split}
\forall \varepsilon>0\forall F\subseteq G\text{ finite symmetric and containing zero} ,\\\forall f:F\rightarrow \Rat \text{ Kat\v etov with respect to }\overline{F}\\ \exists g\in G\forall a\in F (|f(a)-\lambda(g-a)|<\varepsilon).
\end{split}
\end{equation}
It is well-known and straightforward to prove using standard arguments that for any $\lambda$ satisfying \eqref{Uryaftercompl} we have that the completion $\overline{(G,\lambda)}$ is isometric to the Urysohn space.

Moreover, an immediate computation gives that \eqref{Uryaftercompl} is a $G_\delta$ condition. Any norm $\lambda$ with which $G$ is isometric to the rational Urysohn space certainly satisfies \eqref{Uryaftercompl}, thus it follow from Theorem \ref{Urysohnnorms} that the condition defines a dense $G_\delta$ set.
\end{proof}
The proofs of Theorem \ref{Urysohnnorms} and Corollary \ref{corollaryUry} also conclude the proof of Theorem \ref{intro_thm1}.
\begin{remark}
We note that there is one example in the literature of a countable abelian group which admits a norm which makes it isometric to the rational Urysohn space, yet it is not unbounded. It is the countable Boolean group (see \cite{Nie2}). The case of exponent $2$ is obviously special and it is open whether other bounded countable abelian groups admit such a norm (see the open problems in \cite{Nie1}, where it is proved that groups of exponent $3$ do not admit such a norm). We conjecture that they do not.
\end{remark}
\section{Generic norms}
For a norm $\lambda$ on $G$ denote by $\overline{(G,\lambda)}$ the completion. We shall call a norm $\lambda$ on a countable abelian group $G$ \emph{generic} if the set $\{\rho\in\mathcal{N}_G:\overline{(G,\lambda)}\cong \overline{(G,\rho)}\}$ is comeager. In other words, a countable abelian group $G$ admits a generic norm if all the norms on $G$ except those coming from a meager set give rise to the same normed group after the completion. It follows from Theorem \ref{thmMeTs} that if $\lambda$ is a generic norm on a countable unbounded abelian group $G$, then $(G,\lambda)$ is extremely amenable.\\

Let us start with the following easy to check and well known observation. If $(G,\lambda)$ is a normed abelian group and $g\in G$ then $\lim_{n\to \infty} \frac{\lambda(n\cdot g)}{n}$ exists and is equal to $\inf_n \frac{\lambda(n\cdot g)}{n}$. Following Niemiec in \cite{Nie1}, by $\mathcal{O}_0$ we denote the class of those abelian normed groups $(G,\lambda)$ such that for all $g\in G$, $\lim_n \frac{\lambda(n\cdot g)}{n}=0$. The next lemma shows that if there is a generic norm $\lambda$ on $G$, then necessarily $(G,\lambda)\in\mathcal{O}_0$.

\begin{lem}\label{N0normlem}
For every countable abelian group $G$ the set $N_0\subseteq \mathcal{N}_G$ of those norms $\lambda$ on $G$ such that $(G,\lambda)\in\mathcal{O}_0$ is dense $G_\delta$.
\end{lem}
\begin{proof}
First we check that $N_0$ is $G_\delta$. We have $\lambda\in N_0$ if and only if for all $g\in G$ and every $\varepsilon>0$ there exists $n\in\Nat$ such that $\frac{\lambda(n\cdot g)}{n}<\varepsilon$, which is certainly a $G_\delta$ condition.

To show that it is dense it suffices to show that the set of bounded norms is dense, since bounded norms clearly belong to $N_0$. That is however immediate. Take some basic open neighborhood of some $\lambda'$ given by a finite set $A\subseteq G$ and some $\varepsilon>0$ (the $\varepsilon$ will be however irrelevant). Let $M=\max\{\lambda'(g):g\in A\}+1$ and define a norm $\lambda$ as $\min\{\lambda',M\}$.
\end{proof}
From now on, we shall call norms $\lambda$ (on some $G$) such that $(G,\lambda)\in\mathcal{O}_0$, $N_0$-\emph{norms}.
\begin{defin}
Let $G$ be a countable abelian group. We call $G$ \emph{infinitely-summed} if $G\cong \bigoplus_{n\in\Nat} G$.

In particular, notice that if non-trivial $G$ is infinitely-summed then it is not finitely generated.
\end{defin}
\begin{remark}
There is a simple, however universal, way how to produce infinitely-summed groups. If $H$ is an arbitrary countable abelian group, then $\bigoplus_{n\in\Nat} H$ is infinitely-summed. Such groups are highly homogeneous: a reader versed in Fra\" iss\' e theory can check that standard Fra\" iss\' e classes of finitely generated abelian groups have Fra\" iss\' e limits that are infinitely-summed. Consider e.g. the Fra\" iss\' e class of all finite abelian groups whose Fra\" iss\' e limit is $\bigoplus_{n\in\Nat} \Rat/\Int$, or the class of all finitely generated torsion-free abelian groups whose Fra\" iss\' e limit is $\bigoplus_{m\in\Nat} \Rat$.

The main feature of infinitely-summed groups that is used in the proof of Proposition \ref{denseprop} is that if $G$ is such a group and $F$ is its finitely generated subgroup, then there exists a finitely generated subgroup $F'\leq G$, isomorphic to $F$ and such that $F\cap F'=\{0\}$.
\end{remark}

Let $G$ be a non-trivial infinitely-summed countable abelian group. Write $G$ as $\bigoplus_n G_n$, where for each $n\in\Nat$, $G_n\cong G$. For each $i\in \Nat$, let $\phi_i:G_1\rightarrow G_i$ be an isomorphism. Fix some infinite sequence $(d'_n)_n$ of $G_1$ such that for every $n$, $d'_{n+1}\notin\langle d'_i:i\leq n\rangle$, and $G_1=\langle d'_n:n\in\Nat\rangle$. Let $D=\{\phi_i(d'_n):i,n\in\Nat\}$, i.e. $D$ generates $G$. Fix also an enumeration $\{d_n:n\in\Nat\}$ of $D$. For every set of integers (finite in all cases in the sequel) $A\subseteq \Nat$, let $F_A\leq G$ be the subgroup $\langle d_i:i\in A\rangle$. Note that for any finite $A\subseteq \Nat$ there exists finite $A'\subseteq \Nat$ disjoint from $A$ and a bijection $\phi:A\rightarrow A'$ which uniquely determines an isomorphism $\bar{\phi}:F_A\rightarrow F_{A'}$ determined by sending $d_i$ to $d_{\phi(i)}$ for $i\in A$. Let $\Phi$ be the set of all isomorphisms between finitely generated subgroups of the form $F_A$ and $F_{A'}$ 
which are determined by some bijection between $A$ and $A'$. Obviously, for most choices of generating sets $D$, not all bijections between two finite subsets $A,A'\subseteq \Nat$ of the same size give rise to isomorphisms between $F_A$ and $F_{A'}$, which do not have to be isomorphic at all. We write $A\equiv A'$ if there does exist a bijection between $A$ and $A'$ which gives rise to an isomorphism between $F_A$ and $F_{A'}$. If we want to specify the isomorphism, we write $A\equiv_\phi A'$, where $\phi$ is the bijection between $A$ and $A'$ and $\bar{\phi}$ the corresponding isomorphism between $F_A$ and $F_{A'}$.

Also, for each finite $A\subseteq \Nat$, let $|\cdot|_A:F_A\rightarrow \Nat$ be the length function (i.e. norm) associated to the generating set $\{d_i,-d_i:i\in A\}$, i.e. the distance from $0$ in the graph metric of the Cayley graph of $F_A$ with $\{d_i,-d_i:i\in A\}$ as a generating set. If there is no danger of confusion then we write just $|\cdot|$ instead of $|\cdot|_A$.

The next definition introduces a certain distance between two norms $\rho$, resp. $\rho'$ on $F_A$, resp. $F_{A'}$, where $A\equiv A'$.
\begin{defin}
Suppose that $A,A'\subseteq \Nat$ are two finite subsets such that $A\equiv_\phi A'$ for some $\phi\in\Phi$. Suppose also that $F_A$ is equipped with a norm $\rho$ and $F_{A'}$ with a norm $\rho'$. Then we say that $(A,\rho)$ and $(A',\rho')$ are \emph{$\phi,\varepsilon$-close} and write $(A,\rho)\equiv_{\phi,\varepsilon} (A',\rho')$ if $$\sup_{f\in F_A\setminus\{0\}}\frac{|\rho(f)-\rho'(\bar{\phi}(f))|}{|f|_A}<\varepsilon.$$
\end{defin}
Note that if $(A,\rho)\equiv_{\phi,\varepsilon} (A',\rho')$ and $(A',\rho')\equiv_{\psi,\delta} (A'',\rho'')$, then $(A,\rho)\equiv_{\psi\circ \phi,\varepsilon+\delta} (A'',\rho'')$.

We shall also need a notion of `closeness' of two subgroups $F_A,F_{A'}\leq G$ with respect to some norm $\lambda$ on $G$ and some $\phi\in\Phi$ such that $A\equiv_\phi A'$.
\begin{defin}
Suppose that $G$ is equipped with a norm $\lambda$. Let $A,A'$ be two finite subsets of naturals such that for some $\phi\in\Phi$ we have $A\equiv_\phi A'$ and let $\varepsilon>0$. Then we write $\delta^\lambda_\phi (A,A')<\varepsilon$ if for every $a\in A$ we have $\lambda(d_a-\bar{\phi}(d_a))=\lambda(d_a-d_{\phi(a)})<\varepsilon$. Notice that this is equivalent to saying that for every non-zero $f\in F_A$ we have $\lambda(f-\bar{\phi}(f))<\varepsilon\cdot |f|_A$. Again, we shall suppress the upper index $\lambda$ from $\delta^\lambda_\phi$ when it is clear from the context.
\end{defin}

Next lemma shows that the condition $(A,\rho)\equiv_{\phi,\varepsilon} (A',\rho')$ is determined on finite sets for $N_0$-norms.
\begin{lem}\label{lem_findet}
Let $A,A'\subseteq \Nat$ be two finite subsets such that $A\equiv_\phi A'$.  Then for any $N_0$-norm $\rho$ on $F_A$ and any $\varepsilon>0$ there exist a finite subset $C\subseteq F_A$ and $\delta>0$ such that for any $N_0$-norm $\rho'$ on $F_{A'}$, if for every $a\in C$ we have $|\rho(a)-\rho'(\bar{\phi}(a))|<\delta$, then $(A,\rho)\equiv_{\phi,\varepsilon} (A',\rho')$.
\end{lem}
\begin{proof}
To simplify the notation, we shall assume that $\phi$ is the identity. Therefore we look for a finite subset $C\subseteq F_A$ and $\delta>0$ such that for any $N_0$-norm $\rho'$ on $F_A$, if for every $a\in C$ we have $|\rho(a)-\rho'(a)|<\delta$, then $\frac{|\rho(x)-\rho'(x)|}{|x|}<\varepsilon$, for every $x\in F_A\setminus\{0\}$.\\

Since $F_A$ is a finitely generated abelian group it is isomorphic to a sum of a finitely generated free abelian group $F$ and a finite abelian group $K$. Suppose first that $F_A$ is $F$, i.e. it is a finitely generated free abelian group, and moreover suppose that the generators $(d_i)_{i\in A}$ are free generators. Any $x\in F_A$ can be thus uniquely written as $\sum_{i\in A} k_i\cdot d_i$. For $i\in A$ denote by $|x|_i$ the absolute value of $k_i$. Now for any $n\in \Nat$
set $A'_n=\{x\in F_A : \forall i\in A (|x|_i\leq n)\}$ and $A_n=\{x\in F_A : |x| \leq n\}$.

We also set (see Claim 3.3 in \cite{Do1}), for any $k\in\Nat$ $$B'_k=\sup_{x\in F_A\setminus\{0\}} \min\{\frac{|x-l\cdot y|}{|x|}:l\in\Nat,y\in A'_k,|l\cdot y|\leq 2|x|\}$$
and
$$B_k=\sup_{x\in F_A\setminus\{0\}} \min\{\frac{|x-l\cdot y|}{|x|}:l\in\Nat,y\in A_k,|l\cdot y|\leq 2|x|\}.$$

\begin{claim}
There exists $K$ and $\delta>0$ such that for every $N_0$-norm $\rho'$ on $F_A$, if $|\rho(x)-\rho'(x)|<\delta$ for every $x\in A_K$, then for every $x\in F_A$ such that $|x|>K$ we have $\frac{\rho'(x)}{|x|}<\varepsilon/2$.
\end{claim}
Once the claim is proved we are done (with the case when $F_A$ is a free abelian group with free generators). Indeed, use $\delta>0$ from the claim and set as $C$ the set $A_K$. Then if $\rho'$ is a $N_0$-norm on $F_A$ which is $\delta$-close to $\rho$ on $C$, then for any $x\in F_A\setminus\{0\}$
\begin{itemize}
\item if $x\in C$, then trivially $\frac{|\rho(x)-\rho'(x)|}{|x|}<\varepsilon$ as we may assume that $\delta<\varepsilon$;
\item if $|x|>K$, then by the claim we have $\frac{\rho(x)}{|x|}<\varepsilon/2$ and $\frac{\rho'(x)}{|x|}<\varepsilon/2$, so $\frac{|\rho(x)-\rho'(x)|}{|x|}\leq \frac{\rho(x)+\rho'(x)}{|x|}<\varepsilon$.

\end{itemize}
\emph{Proof of the claim.} Set $M=\max\{\rho(d_i):i\in A\}$. It has been proved in \cite{Do1} that the sequence $(B'_k)_k$ converges and in fact we have $\lim_{k\to\infty} B'_k=0$ (see Claim 3.3 in \cite{Do1}). Therefore there exists $k'\in\Nat$ so that $B'_k<\varepsilon/(4M+\varepsilon/8)$. Clearly there is $k$ so that we have $A_{k'}\subseteq A'_{k'}\subseteq A_k$, which yields that $B_k<\varepsilon/(4M+\varepsilon/8)$. We may and do assume that $\varepsilon/(4M+\varepsilon/8)<1/2$.

Now let $K'>k$ be such that for every $y\in A_k$ we have $$\frac{\rho(m\cdot y)}{m}<\varepsilon/32$$ for every $m\geq K'$, which is possible since $\rho$ is an $N_0$-norm. Finally, let $K>2K'\cdot k$ be such that for every $x\in A_{K'\cdot k}$ we have $\frac{\rho(x)}{K}<\varepsilon/16$.\\

We claim that $K$ and $\delta=\varepsilon/32$ are as desired. Take any $N_0$-norm $\rho'$ on $F_A$ which is $\varepsilon/32$-close to $\rho$ on $A_K$. Take any $x\notin A_K$.
We may find some $y\in A_k$ and non-negative $l$, such that $l|y|\leq 2|x|$, and $\frac{|x-l\cdot y|}{|x|}<\varepsilon/(4M+\varepsilon/8)$. We claim that $l\geq K'$ since otherwise $|l\cdot y|<K'\cdot k$, so $|x-l\cdot y|\geq |x|/2$, so $\frac{|x-l\cdot y|}{|x|}\geq 1/2$.

We have $$\frac{\rho'(x)}{|x|}\leq \frac{\rho'(l\cdot y)+\rho'(x-l\cdot y)}{|x|}$$ where $$\frac{\rho'(x-l\cdot y)}{|x|}\leq \frac{(M+\varepsilon/32)|x-l\cdot y|}{|x|}< (M+\varepsilon/32)\varepsilon/(4M+\varepsilon/8)= \varepsilon/4.$$ Note that the first inequality follows from the fact that $\rho(d_i)\leq M+\varepsilon/32$, for $i\in A$, since $\rho(d_i)\leq M$, for $i\in A$, by the definition of $M$ and $|\rho(d_i)-\rho'(d_i)|<\varepsilon/32$, for $i\in A$.

So it suffices to show that $\frac{\rho'(l\cdot y)}{|x|}<\varepsilon/4$. Write $l$ as $tK'+r$, where $0<r<K'$. Note that $t\geq 1$. Then we have $$\frac{\rho'(l\cdot y)}{|x|}\leq \frac{t\cdot \rho'(K'\cdot y)+\rho'(r\cdot y)}{|x|}\leq \frac{2t\cdot \rho'(K'\cdot y)}{t\cdot |K'\cdot y|}+\frac{\rho'(r\cdot y)}{K}.$$
Note that $$\frac{2t\cdot \rho'(K'\cdot y)}{t\cdot |K'\cdot y|}\leq 2\frac{\rho'(K'\cdot y)}{K'}\leq 2\frac{\rho(K'\cdot y)+\varepsilon/32}{K'}< 2(\varepsilon/32+\varepsilon/32)= \varepsilon/8$$ by the definition of $K'$ and since $\rho'$ is $\varepsilon/32$-close to $\rho$ on $A_K$, and $$\frac{\rho'(r\cdot y)}{K}\leq \frac{\rho(r\cdot y)+\varepsilon/32}{K}\leq \varepsilon/16+\varepsilon/32<\varepsilon/8$$ by the definition of $K$ and again
since $\rho'$ is $\varepsilon/32$-close to $\rho$ on $A_K$.\\

Now we suppose that $F_A$ is in general a sum of a finitely generated free abelian group $F$ and a finite group $K$. Let us however assume that the given generators of $F_A$ are free generators of $F$ together with all the non-zero elements of $K$. Denote the length function on $F$ by $|\cdot|_F$ and the length function on $F_A$ determined by the generators above by $|\cdot|'$. First we apply the result above for $F$ with its free generators, $\rho\upharpoonright F$ and $\varepsilon/4$. We get some corresponding finite $C'\subseteq F$ and $\delta>0$ such that $\delta<\varepsilon/2$. Now let $M=(\max \rho(K))+\delta$. Take some $k\in\Nat$ such that $C'\subseteq C=\{x\in F_A:|x|'\leq k\}$ and $M/k<\varepsilon/8$. We claim that now $C\subseteq F_A$ and $\delta>0$ are as desired. Let $\rho'$ be an arbitrary norm on $F_A$ such that $|\rho(a)-\rho'(a)|<\delta$ for all $a\in C$. Take any $z\in F_A\setminus\{0\}$. If $|z|'\leq k$, then $|\rho(z)-\rho'(z)|<\delta$, so $\frac{|\rho(z)-\rho'(z)|}{|z|'}<\varepsilon$, since $\delta\leq \varepsilon/2$. So suppose that $|z|'>k$ and write $z$ as $x+u$, where $x\in F$ and $u\in K$. We have $|x+u|'-|x|_F\in \{0,1\}$ and $$\frac{|\rho(x+u)-\rho'(x+u)|}{|x+u|'}\leq \frac{|\rho(x)-\rho'(x)|+\rho(u)+\rho'(u)}{|x|_F+1}\leq$$ $$\frac{|\rho(x)-\rho'(x)|}{|x|_F+1}+\frac{\rho(u)+\rho'(u)}{k}< \varepsilon/4+2\varepsilon/8+\delta<\varepsilon.$$

Now suppose that the generators of $F_A$ are arbitrary. It is a well-known and easy to observe that the length function $|\cdot|_A$ and the length function $|\cdot|'$ from the paragraph above are bi-Lipschitz equivalent. In particular, there is some $L$ such that $|\cdot|_A\leq L|\cdot|'$. Thus we may apply the result from the paragraph above for $F_A$ with the generators above and with $\varepsilon/L$ to obtain the 
result for $F_A$ with generators $\{d_a:a\in A\}$ and with $\varepsilon$. 
\end{proof}

We present one more definition of a certain easily definable norm.
\begin{defin}
Let us denote a norm on some $F_C$, for $C$ finite, \emph{finitely generated} if it is obtained as an extension using Lemma \ref{extnorm_lem} of some partial norm defined on a finite symmetric subset of $F_C$ containing zero.

Moreover, call a norm on some $F_C$, where again $C$ is finite, \emph{bounded finitely generated}, if it is a bounded norm which is defined as the minimum over a finitely generated norm and a positive constant. This constant is called a bound of the norm.

Note that there are only countably many rational bounded finitely-generated norms on a fixed countable abelian group, i.e. norms that are defined as a minimum over a rational bound and a rational finitely generated norm. We shall call them \emph{BRFG norms}.
\end{defin}
\begin{claim}\label{claim1}
Let $C\subseteq \Nat$ be finite and let $\rho$ be an $N_0$-norm on $F_C$. Then for every $\varepsilon>0$ there exists a BRFG norm $\rho_R$ on $F_C$ such that $(C,\rho)\equiv_{\mathrm{id},\varepsilon} (C,\rho_R)$.
\end{claim}
\begin{proof}[Proof of Claim \ref{claim1}.]
Fix $C$, an $N_0$-norm $\rho$ on $F_C$ and $\varepsilon>0$. We use Lemma \ref{lem_findet} to find a finite set $F\subseteq F_C$ and $\delta>0$ such that for any $N_0$-norm $\rho'$ on $F_C$, if $|\rho(f)-\rho'(f)|<\delta$, for every $f\in F$, then $(C,\rho)\equiv_{\mathrm{id},\varepsilon} (C,\rho')$. We may suppose that $F$ is finite symmetric containing zero. Then we use Lemma \ref{ratpartnorm} to find a partial rational norm $\rho'_R$ on $F$ such that $|\rho(f)-\rho'_R(f)|<\delta$, for every $f\in F$. Then we take the extension of $\rho'_R$ to the whole $F_C$, obtained by Lemma \ref{extnorm_lem}. Finally we bound this extension by $\max_{f\in F} \rho'_R(f)$. This is the desired BRFG norm $\rho_R$.
\end{proof}
We now restate Theorem \ref{intro_thm2} here for the convenience of the reader and start with its proof.
\begin{thm}\label{mainthm}
Let $G$ be an infinitely-summed group. Then $G$ admits a generic norm $\lambda$.
\end{thm}
\begin{remark}
We do not know if there is a countable abelian group which is not infinitely-summed and which admits a generic norm.
\end{remark}

From now on, fix a non-trivial infinitely-summed group $G$, the enumerated set of generators $D=\{d_n:n\in\Nat\}$ as above, and the set of bijections $\Phi$ as above. 

Let $\mathcal{G}\subseteq \mathcal{N}$ denote the set of all $N_0$-norms $\lambda$ on $G$ satisfying the following condition:
\begin{equation}\label{Gdeltacond}
\begin{split}
\forall\varepsilon>\varepsilon'>0, \forall A_0\subseteq A\subseteq \Nat \text{ finite }\forall \rho_A\text{ BRFG norm on } F_A\\ \text{if }(A_0,\rho_A)\equiv_{\mathrm{id},\varepsilon} (A_0,\lambda)\text{ then } \exists A'\subseteq \Nat, \exists \phi\in\Phi\text{ such that}\\   (A,\rho_A)\equiv_{\phi,\varepsilon'} (A',\lambda)\text { and }\delta^\lambda_\phi(A_0,\phi[A_0])<\varepsilon.
\end{split}
\end{equation}

In order to prove Theorem \ref{mainthm}, we shall prove that $\mathcal{G}$ is dense $G_\delta$ and that for any $\lambda,\rho\in\mathcal{G}$ we have $\overline{(G,\lambda)}=\overline{(G,\rho)}$. That will give that $G$ admits a generic norm. Moreover in case $G$ is unbounded, we show that there is $\lambda\in \mathcal{G}$ such that $(G,\lambda)$ is isometric to the rational Urysohn space. That will complete the proof of Theorem \ref{mainthm}.

The first step showed in the next lemma is simple.
\begin{lem}\label{isGdelta_lem}
$\mathcal{G}$ is $G_\delta$.
\end{lem}
\begin{proof}
Being $N_0$-norm is a $G_\delta$ condition by Lemma \ref{N0normlem}, so we check that the condition \eqref{Gdeltacond} is also $G_\delta$.

That follows from the following observations. Fix $\varepsilon>\varepsilon'>0$, $A_0\subseteq A\subseteq \Nat$ finite and a BRFG norm $\rho_A$ on $F_A$. Then the implication from \eqref{Gdeltacond}, after rewriting it as a disjunction, gives a union of the following two sets $$\{\lambda\in\mathcal{N}:\neg (A_0,\rho_A)\equiv_{\mathrm{id},\varepsilon} (A_0,\lambda)\}$$ and $$\{\lambda\in\mathcal{N}:\exists A'\subseteq \Nat \exists \phi\in\Phi\; ((A,\rho_A)\equiv_{\phi,\varepsilon'} (A',\lambda)\wedge \delta^\lambda_\phi(A_0,\phi[A_0])<\varepsilon)\}.$$

Using Lemma \ref{lem_findet}, which says that the relation $\equiv_{\phi,\varepsilon}$ is determined on finite subsets for $N_0$-norms, we get that the first set is closed and the second one is open. In particular, both are $G_\delta$ since $\mathcal{N}$ is a Polish space, so closed sets are $G_\delta$. Since the union of two $G_\delta$ sets is again $G_\delta$ we get that for fixed $\varepsilon>\varepsilon'>0$, $A_0\subseteq A\subseteq \Nat$ finite and a BRFG norm $\rho_A$ on $F_A$ the condition \eqref{Gdeltacond} determines a $G_\delta$ set.

Finally one can check that all the universal quantifiers in \eqref{Gdeltacond} can be taken over countable sets which shows that \eqref{Gdeltacond} defines a $G_\delta$ subset.
\end{proof}
Next we want to show that all the norms from $\mathcal{G}$ give rise to the same normed group after the completion. Note that the condition \eqref{Gdeltacond} is similar to the condition on vector space norms which give rise to the Gurarij space, the separable Banach space of almost universal disposition constructed by Gurarij in \cite{Gu}. The following proposition is thus similar to the main result of \cite{KuSo} where the authors prove the uniqueness of the Gurarij space.
\begin{prop}
For any two $\lambda,\lambda'\in\mathcal{G}$ we have that $\overline{(G,\lambda)}$ and $\overline{(G,\lambda')}$ are isometrically isomorphic.
\end{prop}
\begin{proof}
Consider two norms $\lambda,\lambda'\in\mathcal{G}$. Let $(i_j)_j$ be an enumeration of $\Nat$ with an infinite repetition.

By induction, for every $j\in \Nat$ we shall construct two finite sequences $(a_i^j)_{i=0}^{2j-1}\subseteq \Nat$ and $(b_i^j)_{i=0}^{2j}\subseteq \Nat$ such that\\
\begin{enumerate}
\item\label{item0} for every $j\in\Nat$, $i_j\in\{a_i^j:i\leq 2j-1\}\cap\{b_i^j:i\leq 2j\}$, i.e. there are $k,k'$ such that $i_j=a_k^j=b_{k'}^j$;
\item\label{item1} for every $j\in\Nat$ there are some $\phi_j\in\Phi$ such that $\phi_j(a_i^j)=b_i^j$, for every $i\leq 2j-1$, and $$((a_i^j)_{i=0}^{2j-1},\lambda)\equiv_{\phi_j,1/2^{2j-1}} ((b_i^j)_{i=0}^{2j-1},\lambda'),$$ and $\psi_j\in\Phi$ such that $\psi_j(b_i^j)=a_i^{j+1}$, for every $i\leq 2j$, and $$((b_i^j)_{i=0}^{2j},\lambda')\equiv_{\psi_j,1/2^{2j}} ((a_i^{j+1})_{i=0}^{2j},\lambda);$$
\item \label{item2} for every $j\in\Nat$ we have $$\delta^\lambda_{\psi_j\circ\phi_j}((a_i^j)_{i=0}^{2j-1},(a_i^{j+1})_{i=0}^{2j-1})<1/2^{2j-1}$$ and $$\delta^{\lambda'}_{\phi_{j+1}\circ\psi_j}((b_i^j)_{i=0}^{2j+1},(b_i^{j+1})_{i=0}^{2j+1})<1/2^{2j}.$$\\
\end{enumerate}
Note that in particular for every $j\in\Nat$ we have 
\begin{equation}\label{limmorph}
\begin{split}
a_i^{j+1}=\psi_j\circ\phi_j(a_i^j), \forall i\leq 2j-1,\\
b_i^{j+1}=\phi_{j+1}\circ\psi_j(b_i^j), \forall i\leq 2j.
\end{split}
\end{equation}

Suppose at first that such sequences have been constructed. Denote by $\Gr$ the completion of $(G,\lambda)$ and by $\Gr'$ the completion of $(G,\lambda')$. By \eqref{item2}, for each $i\in\Nat$ we have that the sequence $(g_i^j)_j$, where $g_i^j=d_{a_i^j}$ for all $i,j$, is Cauchy in $(G,\lambda)$, thus it has the limit, denoted by $g_i$, in $\Gr$. Analogously by \eqref{item2}, for each $i\in\Nat$ the sequence $(h_i^j)_j$, where $h_i^j=d_{b_i^j}$ for all $i,j$, is Cauchy in $(G,\lambda')$ and we denote by $h_i$ the limit in $\Gr'$. We claim that $\langle(g_i)_i\rangle$ is a dense subgroup in $\Gr$ and $\langle(h_i)_i\rangle$ is a dense subgroup in $\Gr'$. We prove the former, the latter is analogous. Since $G$ is dense in $\Gr$ it suffices to show that for any $g\in G$ and any $\varepsilon>0$ there exists $g'\in\langle (g_i)_i\rangle$ such that $\lambda(g'-g)<\varepsilon$. Take some finite $C\subseteq \Nat$ such that $g\in F_C$ and let $k=|g|_C$. There exists $N$ such that for every $j\geq N$ and $i\leq 2j-
1$ we have $$\lambda(g_i^j-g_i)<\varepsilon/k.$$ Also, by \eqref{item0}, for each $c\in C$ we can find $i_c$ and $j_c\geq N$ such that $g_{i_c}^{j_c}=d_c$. Since $|g|_C=k$, $\lambda(g_{i_c}^{j_c}-g_{i_c})<\varepsilon/k$, for every $c\in C$, it follows there is an element $g'\in\langle g_{i_c}:c\in C\rangle$ such that $\lambda(g'-g)<\varepsilon$, and the claim is proved.

Next we claim that the map sending $g_i$ to $h_i$, for each $i\in\Nat$, can be extended to an isometric isomorphism $\Psi:\langle g_i:i\in\Nat\rangle\rightarrow \langle h_i:i\in\Nat\rangle$. Take a finite subset $S\subseteq \Nat$ and integers $(k_i)_{i\in S}\subseteq \Int$. Then $$|\lambda(\sum_{i\in S} k_i\cdot g_i)-\lambda'(\sum_{i\in S} k_i\cdot h_i)|=\lim_{j\to\infty} |\lambda(\sum_{i\in S} k_i\cdot g^j_i)-\lambda'(\sum_{i\in S} k_i\cdot h^j_i)|\leq$$ $$\lim_{j\to\infty} \sum_{i\in S} |k_i|/2^{2j-1}=0.$$  The first equality follows from the definition, the second inequality follows from \eqref{item1}.

It follows that we may uniquely extend $\Psi$ to $\Gr$, which we shall still denote by $\Psi$ and which is an isometric isomorphism between $\Gr$ and $\Gr'$. It remains to find the sequences.\\

We will proceed by induction. We show the first odd and even steps of the induction and then the general odd and even steps of the induction.

Set $a_1^1=i_1$. Then by Claim \ref{claim1} there exists a BRFG norm $\rho_1$ on $F_{\{a_1^1\}}$ such that $(\{a_1^1\},\lambda)\equiv_{\mathrm{id},1/4} (\{a_1^1\},\rho_1)$. By \eqref{Gdeltacond}, using $A_0=\emptyset$ and $A=\{a_1^1\}$, $\varepsilon'=1/4$, $\varepsilon$ arbitrary bigger than $\varepsilon'$ and $\rho_1$, there exists $b_1^1$ such that $(\{a_1^1\},\rho_1)\equiv_{\phi_1,1/4} (\{b_1^1\},\lambda')$ for some $\phi_1\in\Phi$, thus $(\{a_1^1\},\lambda)\equiv_{\phi_1,1/2} (\{b_1^1\},\lambda')$. This finishes the first odd step.

Now if $b_1^1=i_1$ then take as $b_2^1$ an arbitrary natural number; otherwise, take $b_2^1=i_1$. By Claim \ref{claim1} there exists a BRFG norm $\rho_2$ on $F_{\{b_1^1,b_2^1\}}$ such that $(\{b_1^1,b_2^1\},\lambda')\equiv_{\mathrm{id},\delta} (\{b_1^1,b_2^1\},\rho_2)$, where $\delta<1/8$ is sufficiently small so that we still have $(\{a_1^1\},\lambda)\equiv_{\phi,1/2} (\{b_1^1\},\rho_2)$.  Then by \eqref{Gdeltacond}, using $A_0=\{b_1^1\}$ and $A=\{b_1^1,b_2^1\}$, $\varepsilon=1/2$, $\varepsilon'=1/8$ and $\rho_2$, there exist $a_1^2,a_2^2\in \Nat$ such that
\begin{itemize}
\item $(\{a_1^2,a_2^2\},\lambda)\equiv_{\phi_2,1/8} (\{b_1^1,b_2^1\},\rho_2)$, for some $\phi_2\in\Phi$,\\ thus  $(\{a_1^2,a_2^2\},\lambda)\equiv_{\phi_2,1/4} (\{b_1^1,b_2^1\},\lambda')$,
\item $\delta_{\phi_2^{-1}\circ\phi_1}(\{a_1^1\},\{a_1^2\})<1/2$; in other words, $\lambda(d_{a_1^1}-d_{a_1^2})<1/2$.
\end{itemize}

Now suppose that we have found sequences $(a_i^{n-1})_{i=0}^{2n-3}$ and $(b_i^{n-1})_{i=0}^{2n-2}$. We shall find $(a_i^n)_{i=0}^{2n-1}$ and $(b_i^n)_{i=0}^{2n}$. Since by assumption we have that $((a_i^{n-1})_{i=0}^{2n-3},\lambda)\equiv_{\phi_{n-1},1/2^{2n-3}}((b_i^n)_{i=0}^{2n},\lambda')$, again using first Claim \ref{claim1} and then \eqref{Gdeltacond} we can find $(a_i^n)_{i=0}^{2n-2}$ and $\psi_{n-1}\in\Phi$ such that $$((b_i^{n-1})_{i=0}^{2n-2},\lambda')\equiv_{\psi_{n-1},1/2^{2n-2}}((a_i^n)_{i=0}^{2n-2},\lambda),$$ and moreover $$\delta_{\psi_{n-1}\circ\phi_{n-1}}((a_i^{n-1})_{i=0}^{2n-3},(a_i^n)_{i=0}^{2n-3})<1/2^{2n-3}.$$ If $i_n\in (a_i^n)_{i=0}^{2n-2}$ then we set $a_{2n-1}^n$ to be any natural number. Otherwise, we set $a_{2n-1}^n=i_n$.

Then analogously, using Claim \ref{claim1} and \eqref{Gdeltacond}, we find $(b_i^n)_{i=0}^{2n-1}$ and $\phi_n\in\Phi$ such that $$((a_i^n)_{i=0}^{2n-1},\lambda)\equiv_{\phi_n,1/2^{2n-1}}((b_i^n)_{i=0}^{2n-1},\lambda'),$$ and moreover $$\delta_{\phi_n\circ\psi_{n-1}}((b_i^{n-1})_{i=0}^{2n-2},(b_i^n)_{i=0}^{2n-2})<1/2^{2n-2}.$$ Again, if $i_n\in (b_i^n)_{i=0}^{2n-1}$ then we set $b_{2n}^n$ to be any natural number. Otherwise, we set $b_{2n}^n=i_n$. This finishes the induction and the proof.
\end{proof}

To finish the proof of Theorem \ref{mainthm} we need to prove that $\mathcal{G}$ is dense. Notice that so far we have not yet even proved that $\mathcal{G}$ is non-empty. The next proposition will do.
\begin{prop}\label{denseprop}
$\mathcal{G}$ is dense.

Moreover, if $G$ is unbounded, then the subset $\{\lambda\in\mathcal{G}:(G,\lambda)\cong_{\mathrm{iso}} \Rat\mathbb{U}\}\subseteq \mathcal{G}$ is dense.
\end{prop}
\begin{proof}
Fix a basic open set in $\mathcal{N}$. It is given by some partial norm $\rho$ on some, without loss of generality, finite symmetric subset $F$ of $G$ containing zero, and some $\varepsilon>0$. We may suppose that $F$ is such that for some finite $C\subseteq \Nat$ we have $\langle F\rangle=F_C$. We use Lemma \ref{ratpartnorm} to get a rational partial norm $\rho'_R$ on $F$ such that $|\rho(f)-\rho'_R(f)|<\varepsilon$ for all $f\in F$. Then we use Lemma \ref{extnorm_lem} to extend it to a rational finitely generated norm on $\langle F\rangle=F_C$ and finally we make it bounded by some rational constant to get a BRFG norm $\rho_R$ on $F_C$ that agrees with $\rho'_R$ on $F$.\\

Let us now enumerate all triples $T_n=(B_n,A_n,\rho_n)$, where $B_n\subseteq A_n\subseteq \Nat$ are finite and $\rho_n$ is a BRFG norm on $F_{A_n}$. Moreover, suppose that there is an infinite repetition of each such a triple in the enumeration.

By induction, we shall construct an increasing sequence of finite sets $(C_n)_n$, i.e. $C_n\subseteq C_m$, for $n<m$, and an increasing sequence of bounded rational norms $(\lambda_n)_n$, i.e. $\lambda_n\subseteq \lambda_m$, for $n<m$, such that
\begin{enumerate}
\item $C_1=C$ and $\lambda_1=\rho_R$;
\item $\bigcup_n C_n=\Nat$;
\item for each $n$, $\lambda_n$ is a norm on $F_{C_n}$;
\item for every $n$, if there are $B'_n\subseteq C_n$ and $\phi'$ such that $$(B'_n,\lambda_n)\equiv_{\phi',1/2^n} (B_n,\rho_n),$$ then there are $A'_n\subseteq C_{n+1}$ and $\phi$ such that $$(A_n,\rho_n)\equiv_\phi (A'_n,\lambda_{n+1}),$$ i.e. $\phi$ induces an isometric isomorphism between $(F_{A_n},\rho_n)$ and $(F_{A'_n},\lambda_{n+1})$, and $$\delta_{\phi\circ\phi'}(B'_n,\phi\circ\phi'[B'_n])<1/2^n.$$\\
\end{enumerate}

We shall now proceed to the induction. The first step has been already done, i.e. we set $C_1=C$ and $\lambda_1=\rho_R$ as obtained from the claim above.

Let us now describe the general step. Suppose we have produced a finite set $C_n$ and a norm $\lambda_n$ on $F_{C_n}$ for $n\geq 1$. Consider now the triple $T_n=(B_n,A_n,\rho_n)$. Suppose that there are some $B'_n\subseteq C_n$ and $\phi'\in\Phi$ such that $\phi'[B'_n]=B_n$ and $(B'_n,\lambda_n)\equiv_{\phi',1/2^n} (B_n,\rho_n)$. There can be at most finitely many such $B'_n$'s. To simplify the notation and proof, we shall suppose there is just one such $B'_n\subseteq C_n$, and actually $B'_n=B_n$, and thus $\phi'=\mathrm{id}$. If there are more such finite subsets of $C_n$, the procedure that follows is repeated (finitely many times). If there is no such a finite subset $B'_n$, then we set $C'_{n+1}=C_n$ and $\lambda'_{n+1}=\lambda_n$ and use the procedure below to extend $C'_{n+1}$ to $C_{n+1}$ and $\lambda'_{n+1}$ to $\lambda_{n+1}$.

Thus we suppose that $B_n\subseteq C_n$ and $(B_n,\lambda_n)\equiv_{\mathrm{id},1/2^n} (B_n,\rho_n)$. Since $G$ is infinitely-summed we can find $A'_n$ such that there is some $\phi\in\Phi$ which is a bijection between $A_n$ and $A'_n$ inducing an isomorphism $\bar{\phi}$ between $F_{A_n}$ and $F_{A'_n}$, and $F_{C_n}\cap F_{A'_n}=\{0\}$. Set $C'_{n+1}=C_n\cup A'_n$. We shall now define a bounded rational norm $\lambda'_{n+1}$ on $F_{C'_{n+1}}$ which extends $\lambda_n$. 

For each $c\in B_n$, set $h_c$ to be $\min\{1/2^n,\lambda_n(d_c)+\rho_n(d_c)\}$. For $x\in F_{C_n}\cup F_{A'_n}\cup\{d_c-\bar{\phi}(d_c),\bar{\phi}(d_c)-d_c:c\in B_n\}$ we set $$\chi (x)=\begin{cases} \lambda_n(x) & \text{if }x\in F_{C_n},\\
\rho_n(y) & \text{if } x=\bar{\phi}(y),y\in F_{A_n},\\
h_c & \text{if } x=\varepsilon(d_c-\bar{\phi}(d_c)),\text{ where }c\in B_n,\varepsilon\in\{1,-1\}.
\end{cases}$$

We use Fact \ref{getnormfact} to get a partial norm $\lambda'_{n+1}$ on $F_{C_n}\cup F_{A'_n}\cup\{d_c-\bar{\phi}(d_c),\bar{\phi}(d_c)-d_c:c\in B_n\}$. Note that it is bounded by some $K$. Also note that for each $c\in B_n$ we have $\lambda'_{n+1}(d_c-\bar{\phi}(d_c))\leq 1/2^n$. We need to check that for each $x\in F_{C_n}\cup F_{A'_n}$ we have $\lambda'_{n+1}(x)=\chi(x)$. Then we could use Lemma \ref{extnorm_lem} again to extend $\lambda'_{n+1}$ to a norm bounded by $K$ on $F_{C'_{n+1}}$ still denoted by $\lambda'_{n+1}$. It will follow that $(A_n,\rho_n)\equiv_\phi (A'_n,\lambda'_{n+1})$ and that $\delta_\phi(B_n,\phi[B_n])\leq 1/2^n$.

So fix some $x\in F_{C_n}\cup F_{A'_n}$. We need to check that for any $x_1,\ldots,x_k\in F_{C_n}\cup F_{A'_n}\cup\{d_c-\bar{\phi}(d_c),\bar{\phi}(d_c)-d_c:c\in B_n\}$ such that $x=\sum_{i=1}^k x_i$ we have $\chi(x)\leq \sum_{i=1}^k \chi(x_i)$.

We have two cases: $x\in F_{C_n}$ and $x\in F_{A'_n}$. We shall treat only the first one, the second is analogous.\\

So we suppose that $x\in F_{C_n}$. Since $G$ is abelian we may suppose that there are $k_1,k_2$ such that $0\leq k_1\leq k_2\leq k$, for every $1\leq i\leq k_1$ we have $x_i\in F_{C_n}$, for every $k_1<i\leq k_2$ we have $x_i\in F_{A'_n}$ and for every $k_2<i\leq k$ we have $x_i\in\{d_c-\bar{\phi}(d_c),\bar{\phi}(d_c)-d_c:c\in B_n\}$. Moreover, for every $k_2<i\leq k$ we may suppose that $\chi(x_i)=1/2^n$. Otherwise, $\chi(x_i)=\lambda_n(d_c)+\rho_n(d_c)$, for some $c\in B_n$, i.e. $x_i$ is equal to $d_c-\bar{\phi}(d_c)$ or $\bar{\phi}(d_c)-d_c$. In that case we would replace $x_i$ by a pair $d_c$, $-\bar{\phi}(d_c)$, resp. $\bar{\phi}(d_c)$, $-d_c$ without increasing the sum $\sum_{i=1}^k \chi(x_i)$. Set $z_1=\sum_{i=1}^{k_1} x_i$, $z_2=\sum_{i=k_1+1}^{k_2} x_i$ and $z_3=\sum_{i=k_2+1}^k x_i$. Since $x=z_1+z_2+z_3$, it follows that $z_3=-z_2+\bar{\phi}^{-1}(z_2)$. Since $(B_n,\lambda_n)\equiv_{\mathrm{id},1/2^n} (B_n,\rho_n)$ we get that $\bar{\phi}^{-1}(z_2)\in F_{B_n}$ and $|\lambda_n(\bar{\phi}^{-1}(z_2))-\rho_n(\bar{\phi}^{-1} (z_2))|<|z_2|/2^n$. Since $$\sum_{i=k_2+1}^k \chi(x_i)\geq |z_2|/2^n,$$
we get that $$\sum_{j=1}^k \chi(x_j)\geq \lambda_n(z_1)+\rho_n(\bar{\phi}^{-1}(z_2))+|z_2|/2^n\geq \lambda_n(x)$$ because $x-z_1=z_2+z_3=\bar{\phi}^{-1}(z_2)$, and we are done.\\

Finally, set $C_{n+1}=C'_{n+1}\cup\{n\}$ and extend $\lambda'_{n+1}$ to a bounded rational norm $\lambda_{n+1}$ on $F_{C_{n+1}}$ arbitrarily.\\

When the induction is finished we get $\lambda=\bigcup_n \lambda_n$ is a norm on $G$. We check that $\lambda\in\mathcal{G}$. First, since it is a direct limit of bounded norms, clearly it is an $N_0$-norm. Now take any $\varepsilon>\varepsilon'>0$, finite subsets $A_0\subseteq A\subseteq \Nat$. Let $\rho_R$ be some BRFG norm on $F_A$ such that $(A_0,\rho_R)\equiv_{\mathrm{id},\varepsilon} (A_0,\lambda)$. Then by the construction, we can find $n$ such that $A\subseteq C_n$ and $T_n=(A_0,A,\rho_R)$, where $1/2^n<\varepsilon$. By the inductive construction, there is some $\phi$ such that $A\equiv_\phi \phi[A]$, $\phi[A]\subseteq C_{n+1}$ and $$(A,\rho_R)\equiv_{\phi,\varepsilon'} (\phi[A],\lambda)$$ and $$\delta_\phi(A,\phi[A])<1/2^n<\varepsilon,$$ and we are done.

Also, since at the beginning the finite symmetric subset $F$ and a partial norm $\rho$ on $F$ were arbitrary, it shows that $\mathcal{G}$ is dense.\\

Finally, we show how to get the ``moreover" part from the statement of the proposition. If $G$ is unbounded then we combine the two induction procedures from this proof and the proof of Theorem \ref{Urysohnnorms} into one. Besides the enumeration of triples $(T_n)_n$ as above, consider also the enumeration $\{(A_i,f_i):i\in\Nat\}$ (again with infinite repetition) of all pairs $(A,f)$, where $A$ is a finite rational metric space and $f:A\rightarrow \Rat$ a rational Kat\v etov function over $A$. Then we divide the induction procedure into odd and even steps. During odd steps, we take care of triples $(T_n)_n$ as above. During even steps, we take care of pairs $(A_n,f_n)_n$ as in the proof of Theorem \ref{Urysohnnorms}, just ensuring that the norm is bounded after every step. It follows that after the induction we get a norm $\lambda\in \mathcal{G}$ such that $(G,\lambda)$ is isometric to the rational Urysohn space.
\end{proof}
\begin{cor}
Let $G$ be an unbounded infinitely-summed countable Abelian group. There exists an Abelian Polish metric group $\Gr$ which is extremely amenable and isometric to the Urysohn space such that for comeager-many norms $\lambda$ on $G$ we have $$\Gr=\overline{(G,\lambda)}.$$

In particular, for every such $G$ there is a norm $\lambda$ such that $(G,\lambda)$ is extremely amenable and isometric to the rational Urysohn space.
\end{cor}
\begin{remark}
Although the results above show that for every infinitely-summed countable Abelian group $G$ there is a corresponding Abelian Polish metric group $\mathbb{G}$, one might ask whether there is actually \emph{a single} generic Abelian Polish metric group $\mathbb{H}$. That is, whether for every infinitely-summed countable Abelian group $G$ and a generic metric $\rho$ on $G$, we have $\overline{(G,\rho)}=\mathbb{H}$.

Clearly, if $G_1$ and $G_2$ are two infinitely-summed countable abelian groups of bounded torsion, where the bounds are different for $G_1$ and $G_2$ respectively, then the corresponding generic metrics cannot give the same group after completion. However, it is reasonable to expect that the groups $\mathbb{G}_\infty(N)$, where $N\in\{0,2,3,4,\ldots\}$, from \cite{Nie1}, are the only generic abelian Polish groups. We remark here that $\mathbb{G}_\infty(0)$, constructed in \cite{Nie1}, is the completion of the Fra\" iss\' e limit of all finite abelian groups with rational norms. For $N\in\{2,3,4,\ldots\}$, $\mathbb{G}_\infty(N)$ is the completion of the Fra\" iss\' e limit of all finite abelian groups of exponent $N$ with rational norms. We refer the reader to \cite{Nie1} for more details.
\end{remark}
\subsection{Extremely amenable universal abelian group} In this section, we observe that the universal abelian Polish group of Shkarin from \cite{Sk}, further investigated by Niemiec in \cite{Nie1}, is extremely amenable, and provide another proof that it is, with its norm, isometric to the Urysohn space. More precisely, we formulate a certain extension property of this group (analogous to the extension property of the Gurarij Banach space) and show that this extension property is a dense $G_\delta$ property. Then we use our results and results of Melleray and Tsankov to show that the extension property defines this group uniquely up to isometric isomorphism, and that the group is extremely amenable and isometric to the Urysohn space.

Since the group is essentially constructed using Fra\" iss\' e theoretic methods, we shall assume here that the reader has a basic knowledge of this area. We refer the reader to Chapter 7 in \cite{Ho} as a reference to the Fra\" iss\' e theory.\\

We note that the group is universal in the sense that every abelian Polish group, or every second-countable abelian Hausdorff group, embeds via topological isomorphism as a subgroup. The group was constructed as a normed group, however it is universal only in the topological sense. It does not contain \emph{isometrically} every separable abelian normed group as a subgroup. A normed group universal in this stronger sense was constructed by the author in \cite{Do1}.\\

Let us describe the construction of the Shkarin's group, further denoted by $\Gr_S$: First one considers the class of all finite abelian groups equipped with rational-valued norms. Then one can check that this class is a Fra\" iss\' e class, thus it has a Fra\" iss\' e limit. It is straightforward to check that the limit is a countable abelian group, denoted by $G_S$, which is algebraically isomorphic to $\bigoplus_\Nat \Rat/\Int$, equipped with a rational norm $\lambda_S$. In particular, $G_S$ is an infinitely-summed unbounded group. Then one takes the completion $\overline{(G_S,\lambda_S)}$ to obtain the group $\Gr_S$.

In the sequel, we shall use the following notation. For two finite groups $G,H$ which are isomorphic via some $\phi$, and which are equipped with norms $\lambda_G$ and $\lambda_H$ respectively, and for some $\varepsilon>0$ we shall write $$(G,\lambda_G)\equiv_{\phi,\varepsilon} (H,\lambda_H)$$ to express that they are $\varepsilon$-isomorphic via $\phi$, i.e. for every $g\in G$ we have $|\lambda_G(g)-\lambda_H(\phi(g))|<\varepsilon$.\\

Consider now the following set $\mathbb{S}$ of norms $\lambda$ on $G_S=\bigoplus_\Nat \Rat/\Int$, i.e. a subset of $\mathcal{N}_{G_S}$:
\begin{equation}\label{isShkarin}
\begin{split}
\forall \varepsilon>0 \forall G_0\leq G_1\leq G_S\text{ finite},\forall \rho\text{ a rational norm on }G_1\\
\text{if }(G_0,\rho)\equiv_{\mathrm{id},\varepsilon} (G_0,\lambda)\text{ then }\exists G'_1, G_0\leq G'_1\leq G_S,\text{ isomorphic to } G_1\\ \text{ via some }\phi\text{ such that } (G'_1,\lambda)\equiv_{\phi,\varepsilon} (G_1,\rho).
\end{split}
\end{equation}

First, it is clear, by the definition of a Fra\" iss\' e limit, that $\lambda_S\in\mathbb{S}$; moreover, that the set $\mathbb{S}$ is dense. Secondly, one can check as in the proof of Lemma \ref{isGdelta_lem} that \eqref{isShkarin} defines a $G_\delta$ subset of $\mathcal{N}_{G_S}$. Thus $\mathbb{S}$ is a dense $G_\delta$ set.

It also follows from the proofs of Shkarin and Niemiec that for any $\lambda\in \mathbb{S}$, the group $\overline{(G_S,\lambda)}$ is also universal replicating the approximation arguments for $\lambda_S$. It is analogous to the case of metrics on a countable set satisfying \eqref{Uryaftercompl} from the proof of Corollary \ref{corollaryUry}. Then applying Theorem \ref{mainthm}, Corollary \ref{Uryaftercompl} and Theorem \ref{thmMeTs} of Melleray and Tsankov we get the following corollary.
\begin{cor}
There exists a generic norm $\lambda$ on $G_S=\bigoplus_\Nat \Rat/\Int$ such that $\overline{(G_S,\lambda)}$ is the Shkarin's group, which is thus extremely amenable and as a metric space with the metric induced by $\lambda$ isometric to the Urysohn space.
\end{cor}
\begin{remark}
On the other hand, the metrically universal abelian group constructed by the author in \cite{Do1} is not generic. This can be immediately seen as the norm on that group cannot be an $N_0$-norm.
\end{remark}
\section{Problems}
The result of Melleray and Tsankov cannot be used to prove that the metrically universal group from \cite{Do1} is extremely amenable. That does not mean it is not possible though. We note that although the results of this paper cannot be directly used to show that the metrically universal group is isometric to the Urysohn space, it was nevertheless proved in \cite{Do1}. So we ask:
\begin{question}
Is the metrically universal group from \cite{Do1} extremely amenable?
\end{question}

Another challenging problem is to investigate similar properties of the spaces of metrics on non-abelian countable groups. There one can distinguish two cases, which coincide in the case of abelian groups: the space of all \emph{continuous} left-invariant metrics (we comment on the `continuity' below) and the space of bi-invariant metrics. In the case of bi-invariant metrics, one can again easily check that for any countable group $G$, the set of all bi-invariant metrics, is a closed subset $\Rea^{G\times G}$, thus a Polish space.

We do not know if there is some countable non-abelian group which admits a generic bi-invariant metric. Also, we do not know whether for some countable non-abelian group the subset of bi-invariant metrics with which this group is extremely amenable is comeager; i.e. we do not know whether it is possible to generalize the Melleray and Tsankov's result to the non-abelian situation.\\

Regarding the general left-invariant metrics, first thing to observe is that while bi-invariant metrics make the group operations continuous, in fact Lispchitz, this is no longer true for general left-invariant metrics. The corresponding general norms on groups (that do not necessarily make the group topological) were considered in the literature, see e.g. \cite{BiOs}. However, in most cases it is reasonable to consider only such metrics that do make the group operations continuous. For a group $G$ and a left-invariant metric $d$ on $G$, the group operations are continuous if and only if for every $g\in G$ and every $\varepsilon>0$ there is $\delta>0$ such that for every $h\in G$, if $\lambda_d(h)<\delta$, then $\lambda_d(g^{-1}\cdot h\cdot g)<\varepsilon$, where $\lambda_d$ is the corresponding norm. Let us call such norms and metrics continuous.

The main problem is that we do not know how to code the continuous norms and metrics on a non-abelian group as a Polish space. Indeed, the straightforward computation gives that they form an $F_{\sigma \delta}$ subset of $\Rea^G$ ($\Rea^{G^2}$), thus not necessarily a space with Polish topology. Even some special subsets of continuous norms such as uniformly discrete norms seem not to be Polish, but rather $F_\sigma$ subsets of $\Rea^G$. A special subclass of uniformly discrete norms, often considered in geometric group theory, that is, the class of all proper norms again seems to be $F_{\sigma\delta}$.

However, we do not exclude the possibility that a better computation reveals that these are Polish spaces --- in a natural way.\\

\noindent{\bf Acknowledgement:} The author was supported by the GA\v CR project 16-34860L and RVO: 67985840, and by the region Franche-Comt\' e.


\begin{thebibliography}{15}
\bibitem{BiOs}
N. H. Bingham, A. J. Ostaszewski, \emph{Normed versus topological groups: dichotomy and duality}, Dissertationes Math. (Rozprawy Mat.) 472 (2010).
\bibitem{CaVe}
P.J. Cameron, A.M. Vershik, \emph{Some isometry groups of the Urysohn space}, Ann. Pure Appl. Logic 143 (2006), no. 1-3, 70-78.
\bibitem{Do1}
M. Doucha, \emph{Metrically universal abelian groups}, Trans. Amer. Math. Soc. 369 (2017), 5981--5998.
\bibitem{Do2}
M. Doucha, \emph{Metrical universality for groups}, Forum Math. 29 (2017), no. 4, 847--872.
\bibitem{Gu}
V.I. Gurarij, \emph{Spaces of universal placement, isotropic spaces and a problem of Mazur on rotations of Banach spaces}, Sibirsk. Mat. Zh. 7 (1966) 1002--1013 (in Russian).
\bibitem{Ho}
W. Hodges, \emph{Model theory. Encyclopedia of Mathematics and its Applications, 42}, Cambridge University Press, Cambridge, 1993.
\bibitem{Ke}
A. Kechris, \emph{Classical descriptive set theory}, Graduate Texts in Mathematics, 156. Springer-Verlag, New York, 1995.
\bibitem{KuSo}
W. Kubi\' s, S. Solecki, \emph{A proof of uniqueness of the Gurarii space}, Israel J. Math. 195 (2013), no. 1, 449--456.
\bibitem{MeTs}
J. Melleray, T. Tsankov, \emph{Generic representations of abelian groups and extreme amenability}, Israel J. Math. 198 (2013), no. 1, 129--167.
\bibitem{Nie1}
P. Niemiec, \emph{Universal valued Abelian groups}, Adv. Math. 235 (2013), 398-449.
\bibitem{Nie2}
P. Niemiec, \emph{Urysohn universal spaces as metric groups of exponent 2}, Fund. Math. 204 (2009), no. 1, 1--6.
\bibitem{Pe}
V. Pestov, \emph{Dynamics of infinite-dimensional groups.} The Ramsey-Dvoretzky-Milman phenomenon. University Lecture Series, 40. American Mathematical Society, Providence, RI, 2006.
\bibitem{Sk}
S. Shkarin, \emph{On universal abelian topological groups}, Mat. Sb.  190  (1999), no. 7, 127--144.
\bibitem{Ur}
P. S. Urysohn, \emph{Sur un espace m\' etrique universel}, Bull. Sci. Math. 51 (1927), 43--64, 74--96.
\end{thebibliography}
\end{document}